\documentclass[a4paper,12pt]{article}
\usepackage[utf8]{inputenc}
\usepackage{a4wide}
\usepackage{latexsym,amsfonts,amsmath,amssymb,mathrsfs,url,amsthm}
\usepackage{mathtools}
\mathtoolsset{showonlyrefs}
\usepackage{color,xcolor,graphicx}
\colorlet{RED}{red}
\usepackage{mathdots}
\usepackage{cancel}
\usepackage{enumitem}
\usepackage{hyperref}

\newtheorem{theorem}{Theorem}[section]
\newtheorem{lemma}[theorem]{Lemma}

\newtheorem{proposition}[theorem]{Proposition}
\newtheorem{definition}[theorem]{Definition}
\newtheorem{corollary}[theorem]{Corollary}

\providecommand{\keywords}[1]
{
  \noindent \small	
  \textbf{Keywords:} #1
}

\newcommand{\bsx}{\boldsymbol{x}}

\newcommand{\bszero}{\boldsymbol{0}}
\newcommand{\bsone}{\boldsymbol{1}}

\newcommand{\NN}{\mathbb{N}}
\newcommand{\RR}{\mathbb{R}}

\allowdisplaybreaks

\newcommand{\Ftwo}{\mathbb{F}_2}
\newcommand{\slice}[2]{[#1 \!:\! #2]}

\title{Exact $\ell^\infty$-separation radius of Sobol' sequences in dimension 2
\thanks{
The work of K.~S. is supported by JSPS KAKENHI Grant Number 24K06857.}}
\author{
Kosuke Suzuki\thanks{Faculty of Science, Yamagata University, 1-4-12 Kojirakawa-machi, Yamagata, 990-8560, Japan (\url{kosuke-suzuki@sci.kj.yamagata-u.ac.jp})}
}
\date{\today}

\begin{document}



\date{\today}

\maketitle

\begin{abstract}
Quasi-uniformity is a fundamental geometric property of point sets, crucial for applications such as kernel interpolation, Gaussian process regression, and space-filling experimental designs.
While quasi-Monte Carlo methods are widely recognized for their low-discrepancy characteristics, understanding their quasi-uniformity remains important for practical applications.
For the two-dimensional Sobol' sequence, Sobol' and Shukhman (2007) conjectured that the separation radius of the first $N$ points achieves the optimal rate $N^{-1/2}$, which would imply quasi-uniformity.
This conjecture was disproved by Goda (2024), who computed exact values of the $\ell^2$-separation radius for a sparse subsequence $N = 2^{2^v-1}$.
However, the general behavior of the Sobol' sequence for arbitrary $N$ remained unclear.
In this paper, we derive exact expressions for the $\ell^\infty$-separation radius of the first $N = 2^m$ points of the two-dimensional Sobol' sequence for all $m \in \mathbb{N}$.
As an immediate consequence, we show that the separation radius of Sobol' points is  $O(N^{-3/4})$, which is strictly worse than the optimal rate $N^{-1/2}$,
revealing that the two-dimensional Sobol' sequence has a suboptimal mesh ratio that grows at least as $N^{1/4}$.
\end{abstract}


\section{Introduction}
Quasi-Monte Carlo (QMC) methods replace random sampling with carefully constructed deterministic point sets for numerical integration over the unit cube; see, e.g., \cite{DKP22,DKS13,DP10,LePi14,N92,SJ94}.
While classical QMC theory focuses on \emph{discrepancy}, which measures deviation from perfect equidistribution,
it does not directly control local spacing between points.
However, many applications require a more geometric notion of uniformity.
In particular, tasks such as scattered data approximation \cite{W05}, Gaussian process regression \cite{T20}, kernel interpolation \cite{Sch95}, and the design of computer experiments \cite{FLS06, PM12, SWN03} benefit from point sets that ensure both well-controlled spacing and coverage. These properties are summarized by the term \emph{quasi-uniformity}.

To formalize quasi-uniformity, two standard geometric parameters are commonly used for a finite point set $Q \subset [0,1]^d$: the \emph{covering radius} $h_p(Q)$ and the \emph{separation radius} $q_p(Q)$.
Specifically, for the $\ell^p$ norm $\|\cdot\|_p$, define
\[
h_p(Q)\ :=\sup_{x\in[0,1]^d}\min_{y\in Q}\|x-y\|_p,
\qquad
q_p(Q)\ :=\dfrac{1}{2} \min_{\substack{x,y\in Q\\x\neq y}}\|x-y\|_p.
\]
Intuitively, $h_p(Q)$ is the smallest radius such that closed $\ell^p$ balls around each point cover $[0,1]^d$, whereas $q_p(Q)$ is the largest radius such that the corresponding open balls around the points do not overlap.
As noted in \cite[Section~6]{N92}, \cite{SS07}, and \cite{PZ23}, this geometric interpretation implies
\begin{equation}\label{eq:ratio}
h_p(Q) \in \Omega(|Q|^{-1/d}), \qquad q_p(Q) \in O(|Q|^{-1/d}).
\footnote{We use the standard asymptotic notations. For two non-negative functions $f,g$,
$f(N)\in \Omega(g(N))$ if there exists a constant $c>0$ such that $f(N) \ge c g(N)$ for all sufficiently large $N$,
$f(N)\in O(g(N))$ if $f(N) \le c g(N)$ for all sufficiently large $N$,
and $f(N) \in \Theta(g(N))$ if $f(N) \in \Omega(g(N))$ and $f(N) \in O(g(N))$.}
\end{equation}
The \emph{mesh ratio}
\[\rho_p(Q) := \dfrac{h_p(Q)}{q_p(Q)}\]
quantifies how close $Q$ is to an ideal packing/covering configuration.
A sequence $(\bsx_n)_{n\ge 0}$ is \emph{quasi-uniform in $\ell_p$} 
if the mesh ratio $\rho_p(Q_N)$, for $Q_N = \{\bsx_0,\dots,\bsx_{N-1}\}$, 
is bounded independently of $N$.
Equivalently, by \eqref{eq:ratio}, both $h_p(Q_N)$ and $q_p(Q_N)$ are $\Theta(N^{-1/d})$.

Given that QMC point sets are natural candidates for generating uniform point sets,
it is important to ask whether they are quasi-uniform.
This question is non-trivial, since quasi-uniformity does not necessarily imply low discrepancy, nor does low discrepancy imply quasi-uniformity for $d \ge 2$ \cite[Section~1.3]{DGLPSxx}.
As noted in \cite{WBG21}, this issue remained open for a long time for classical QMC constructions.
This situation changed with Goda's influential result \cite{G24a}, which shows that the two-dimensional Sobol' sequence is not quasi-uniform. 

The Sobol' sequence \cite{So67} is widely used in practice
due to its efficient digital construction, extensibility to arbitrary sample sizes, and effectiveness in high-dimensional integration.
It is available in standard software packages, such as Python's QMCPy \cite{QMCpy} and MATLAB's Statistics and Machine Learning Toolbox. 
Sobol' and Shukhman \cite{SS07} conjectured that the separation radius of the first $N$ points of the $d$-dimensional Sobol' sequence behaves like $N^{-1/d}$, which would imply quasi-uniformity. 
For $d=2$, Goda \cite{G24a} showed that this conjecture fails for the sparse subsequence of the form $N = 2^{2^w-1}$, and the case $N = 2^{2^w}$ was subsequently analyzed in \cite{DGSxx}. 
Nevertheless, a complete characterization of the separation radius in the primary case $N=2^m$ remained open,
and closing this gap is the aim of the present work.

\medskip
\noindent\textbf{Main results.}
Let $Q_N$ denote the first $N$ points of the two-dimensional Sobol' sequence (Definition~\ref{def:Sobol}).
Our main contributions are as follows:
\begin{itemize}
\item 
We provide an exact formula for the $\ell^\infty$-separation radius of $Q_{2^m}$ for all $m$.
\item
We prove that $q_{\infty}(Q_N) \in O(N^{-3/4})$ and $\rho_{\infty}(Q_N) \in \Omega(N^{1/4})$.
\end{itemize}
These results are summarized in the following theorem and its corollaries.

\begin{theorem}\label{thm:main}
Let $m \in \NN$ and $N=2^m$.
Let $Q_{N}$ be the first $N$ points of the two-dimensional Sobol' sequence.
If $m=2^v$ or $2^v-1$ for some $v \in \NN$, then we have
\[
q_{\infty}(Q_N) = 2^{-m-1}.
\]
Otherwise, decomposing $m$ as
$m = 2^{v} + 2^{w} + c$ with integers $v > w \ge 0$ and $0 \le c < 2^w$
(note that the case $(w,c)=(v-1,2^{v-1}-1)$ is excluded since this falls into the first case as $m=2^{v+1}-1$),
then
\[
q_{\infty}(Q_N) = 2^{-{2^{v}}-2^{w}}.
\]
\end{theorem}

The following corollaries show that the separation radius of the two-dimensional Sobol' sequence 
decays as $O(N^{-3/4})$, which is a factor $N^{-1/4}$ smaller than the optimal order $\Theta(N^{-1/2})$. 
As a consequence, the Sobol' sequence has a suboptimal mesh ratio for all $N$. 
Since all $\ell^p$ norms on $\RR^2$ are equivalent, the asymptotic orders stated below remain valid for any $p \in [1,\infty]$.

\begin{corollary}\label{cor:net}
For any $m \in \NN$, we have
\begin{align}
q_{\infty}(Q_{2^m}) &\le 2^{-3m/4-5/4}, \label{eq:cor-net-general}\\
q_{\infty}(Q_{2^m}) &\le 2^{-3m/4-3/2} \quad \text{for } m \neq 1,5, \label{eq:cor-net-strong}
\end{align}
with equality in the second inequality if $m = 2^v-2$ for some $v \ge 2$.
\end{corollary}

\begin{corollary}\label{cor:limsup}
For any integer $N \ge 2$, the following bounds hold:
\begin{align}
q_\infty(Q_N) &\le C_1 N^{-3/4}, \label{eq:limsup}\\
\rho_\infty(Q_N) &\ge C_2 N^{1/4}, \label{eq:meshratio}
\end{align}
where $C_1 = C_2 = 2^{-1/2}$. 
Moreover, for $N \ge 64$, the constants can be improved to $C_1 = 2^{-3/4}$ and $C_2 = 2^{-1/4}$.
\end{corollary}

Sobol' sequences are most commonly and effectively used when the number of points is a power of two ($N=2^m$).
Therefore, our results address the primary use case and are sufficient for the vast majority of practical applications.
Despite this practical sufficiency, the exact values of the separation radius for arbitrary $N$ remain a natural mathematical question,
which we investigate in Section~\ref{sec:arbitraryN}.

\medskip
\noindent\textbf{Related work.}
In recent years, the quasi-uniformity of QMC point sets has been the subject of intensive investigation. 
There are two major classes of QMC constructions: digital nets and sequences \cite{DP10,N92}, and lattice point sets (including their infinite analogue, Kronecker sequences) \cite{DKP22,SJ94}. 

For lattice point sets, quasi-uniformity has been actively studied. 
In one dimension, the separation radius of Kronecker sequences is completely characterized in \cite{G24b}. 
More general results in higher dimensions are provided in \cite{DGLPSxx}, including 
bounded mesh ratios for two-dimensional Fibonacci lattices, 
existence results for $d$-dimensional lattice rules, 
and explicit constructions for $d$-dimensional Kronecker sequences. 

Turning to digital nets and sequences,
the one-dimensional van der Corput sequence in base $b$ is known to be quasi-uniform since its first $b^m$ points are $\{i/b^m \mid 0 \le i < b^m\}$.
In higher dimensions, the covering radius, often referred to as \emph{dispersion},
has been extensively studied \cite[Chapter~6]{N92}. 
In particular, for $(t,d)$-sequences in base $b$---well-known examples include the Sobol', Faure, and Niederreiter sequences---, the covering radius is known to attain the optimal order $\Theta(N^{-1/d})$ for any dimension $d$.
Thus, the problem of establishing quasi-uniformity reduces to verifying whether the separation radius also scales as $N^{-1/d}$.

For $d=2$, as stated, the Sobol' sequence is not quasi-uniform \cite{G24a}. 
The separation radius of several two-dimensional digital nets was studied in \cite{GHSK08,GK09}. 
Numerical experiments therein suggest that the Larcher--Pillichshammer nets \cite{LP01} are quasi-uniform. 
This was theoretically proved by Dick, Goda, and Suzuki \cite{DGSxx}, 
who introduced an algebraic criterion for \emph{well-separated} digital nets.
To our knowledge, this remains the only explicit construction of low-discrepancy and quasi-uniform digital nets for $d \ge 2$.
The paper \cite{DGSxx} also shows the non-optimality of the separation radius for some two-dimensional digital nets and Fibonacci polynomial lattices,
as well as for $b$-dimensional Faure sequences in prime base $b$. 

From a different perspective, Pronzato and Zhigljavsky \cite{PZ23} constructed quasi-uniform infinite sequences via a greedy packing algorithm, ensuring a mesh ratio of at most $2$ for the first $N$ points, $N \ge 2$. However, these sequences do not necessarily maintain low discrepancy in dimensions $d \ge 2$.

\medskip
\noindent\textbf{Organization.}
Preliminaries and notation are collected in Section~\ref{sec:prelim}. 
Section~\ref{sec:proofs} provides the necessary lemmas, the proof of Theorem~\ref{thm:main}, and the derivation of Corollaries~\ref{cor:net} and~\ref{cor:limsup}.
Results for arbitrary $N$ are given in Section~\ref{sec:arbitraryN}.

\section{Preliminaries}\label{sec:prelim}

\medskip
\noindent\textbf{Notation.}
Throughout this paper, let $\Ftwo$ denote the finite field of order $2$, 
$\NN$ the set of positive integers, and $\NN_0 := \NN \cup \{0\}$.
Addition in $\Ftwo$ or $\Ftwo^m$ is denoted by $\oplus$.
We write $\bszero_m \in \Ftwo^m$ for the zero vector and $\bsone_m \in \Ftwo^m$ for the all-ones vector. 
The subscript $m$ will be omitted whenever it does not cause confusion.

For an integer $0 \le n < 2^m$ with binary expansion
\[
n = n_1 + 2 n_2 + \cdots + 2^{m-1} n_{m},
\]
we define
\[
\vec{n} = (n_1, n_2, \ldots, n_{m})^\top \in \Ftwo^m,
\]
where $n_1, \dots, n_{m} \in \{0,1\}$ are identified with elements of $\Ftwo$.

For a vector $\vec{z} = (z_1, \dots, z_m)^\top \in \Ftwo^m$, 
we denote by $\vec{z}[i] \in \{0,1\}$ its $i$th component $z_i$, 
and for $i<j$ we define the slice $\vec{z}\slice{i}{j} = (z_i, \dots, z_j)^\top$.
Finally, we set
\[
\phi(\vec{z}) = \vec{z}[1] 2^{-1} + \cdots + \vec{z}[m] 2^{-m}.
\]

For a matrix $P = (P_{ij})_{1 \le i,j \le m}$,  
we use the notation $P[i][j] := P_{ij}$ for convenience, and define
\[
P\slice{x}{y}\slice{z}{w}
  := (P[i][j])_{x \le i \le y, \; z \le j \le w},
\]
which represents the submatrix of $P$ consisting of rows $x$ through $y$ and columns $z$ through $w$.

\medskip
\noindent\textbf{Pascal matrix.}
The (upper triangular) Pascal matrix $P_m \in \Ftwo^{m \times m}$ is defined by
\begin{equation}\label{eq:Pascal-def}
P_m[i][j] \equiv \binom{j-1}{i-1} \pmod 2, \quad 1 \le i,j \le m.    
\end{equation}
The subscript $m$ will be omitted whenever it does not cause confusion.

\medskip
\noindent\textbf{Two-dimensional Sobol' sequence.}
The two-dimensional Sobol' sequence is defined as follows. 
For the definition in general dimension $d$, we refer the reader to \cite[Chapter~8]{DP10}.

\begin{definition}\label{def:Sobol}
The two-dimensional Sobol' sequence $(\bsx_n)_{n \in \NN}$ is a sequence of points in $[0,1)^2$, 
where the $n$th point is given by
\[
\bsx_n := (\phi(\vec{n}), \phi(P_m \vec{n})),
\]
where $m \in \NN$ is chosen such that $n < 2^m$.
This definition does not depend on the choice of $m$:
increasing $m$ simply pads $\vec{n}$ with leading zeros,
and since $P_m$ is upper triangular,
$\phi(\vec{n})$ and $\phi(P_m \vec{n})$ remain unchanged.
\end{definition}

Note that some implementations generate Sobol' sequences in Gray code order rather than in the radical inverse order used here.
This difference affects only the ordering of the points; for $N = 2^m$, the resulting point sets are identical.

This construction is generalized to the notions of \textit{digital nets} and \textit{digital sequences}. 
The uniformity of these point sets is usually measured by a quantity called the $t$-value (see \cite[Chapter~4]{N92} or \cite[Chapter~4]{DP10}).
It is known that the $t$-value of the two-dimensional Sobol' sequence attains the best possible value, namely zero \cite[Section~8.1]{DP10}.
This discussion can be formulated rigorously in the following proposition.

\begin{proposition}\label{prop:0ms}
Let $m \in \NN$ and let $Q_{2^m}$ denote the set of the first $2^m$ points of the two-dimensional Sobol' sequence.
Then, for any rectangle of the form
\[
\left[\frac{a}{2^k}, \frac{a+1}{2^k}\right) \times \left[\frac{b}{2^l}, \frac{b+1}{2^l}\right), 
\quad (a,b,k,l \in \NN_0,\, 0 \le a < 2^k,\, 0 \le b < 2^l)
\]
with $k+l = m$, there is exactly one point from $Q_{2^m}$ contained in the rectangle.
\end{proposition}

\section{Proofs}\label{sec:proofs}
\subsection{Lemmas}
We make heavy use of the properties of the Pascal matrix.
In particular, its entries modulo $2$ can be characterized using Lucas's theorem.
Specifically, for integers $0 \le p, q < 2^m$, we have
\begin{equation}\label{eq:Lucas}
\binom{p}{q} \equiv \prod_{i=1}^m \binom{\vec{p}[i]}{\vec{q}[i]} \pmod{2}.
\end{equation}
Using this result, we can establish the following properties.

\begin{lemma}\label{lem:Pascal}
Let $P$ be the Pascal matrix. 
Let $v > w \ge 0$ and $i \ge 0$ be integers, and set $V := 2^v$ and $W := 2^w$. 
Then the following hold:
\begin{enumerate}[label=\textup{(\roman*)}]
\item \label{item:Pascal1}
$P[i][W] = 1 \iff 1 \le i \le W$.

\item \label{item:Pascal11}
$P[i][V+W] = 1 \iff 1 \le i \le W \text{ or } V+1 \le i \le V+W$.

\item \label{item:Pascal12}
For any $\vec{p} \in \Ftwo^m$ with $m \le 2V-1$, 
we have $(P_m \vec{p})[V] = \vec{p}[V]$ and $(P_m \vec{p})[V+W] = \vec{p}[V+W]$.

\item \label{item:Pascal13}
For any $\vec{p} \in \Ftwo^m$ with $m \le 2V-2$, 
we have $(P_m \vec{p})[V-1] = \vec{p}[V-1] \oplus \vec{p}[V]$.

\item \label{item:Pascal2}
$P_V
= P\slice{1}{V}\slice{V+1}{2V}
= P\slice{V+1}{2V}\slice{V+1}{2V}$.


\item \label{item:Pascal4}
$(P_{W} \bsone_W)[i] = 1 \iff i = W$.

\item \label{item:Pascal5}
$(P_{V+W} \bsone_{V+W})[i] = 1 \iff i = W, V \text{ or } V+W$.

\item \label{item:Pascal6}
$(P_{V-1} \bsone_{V-1})[i] = \bsone_{V-1}$.
\end{enumerate}
\end{lemma}

\begin{proof}
Items \ref{item:Pascal1}--\ref{item:Pascal2} follow directly from Lucas's theorem \eqref{eq:Lucas}.

To prove the remaining items, we note that for any $m \in \NN$ and $1 \le i \le m$, 
the hockey-stick identity implies
\[
(P_m \bsone_m)[i]
\equiv \bigoplus_{j=1}^m \binom{j-1}{i-1}
\equiv \binom{m}{i} \pmod{2}.
\]
Using this fact, 
Items \ref{item:Pascal4}--\ref{item:Pascal6}
also follow from Lucas's theorem.
\end{proof}

The following results assert that the binary representations of two close points are related.

\begin{lemma}\label{lem:xor_of_near_vec}
Let $\ell,m,p,q$ be integers with $2 \le \ell \le m$ and $0 \le p,q < 2^m$.
Assume that $0 \le \phi(\vec{q}) - \phi(\vec{p}) < 2^{-\ell+1}$.
Then one of the following holds:
\begin{enumerate}[label=\textup{(\roman*)}]
\item $\vec{p}\slice{1}{\ell-1} = \vec{q}\slice{1}{\ell-1}$.

\item There exists $1 \le k \le \ell-1$ such that all of the following conditions hold:
\begin{enumerate}[label=\textup{(\alph*)}]
    \item $\vec{p}\slice{1}{k-1} = \vec{q}\slice{1}{k-1}$, \label{enu:diff1}
    \item $\vec{p}[k] = 0$, $\vec{q}[k] = 1$, \label{enu:diff2}
    \item $\vec{p}\slice{k+1}{\ell-1} = \bsone$, $\vec{q}\slice{k+1}{\ell-1} = \bszero$, \label{enu:diff3}
    \item $\vec{p}[\ell] \ge \vec{q}[\ell]$, \label{enu:diff4}
    \item $\vec{p}\slice{\ell}{m} \neq \vec{q}\slice{\ell}{m}$. \label{enu:diff5}
\end{enumerate}
\end{enumerate}
\end{lemma}

\begin{proof}
We assume that $\vec{p}\slice{1}{\ell-1} \neq \vec{q}\slice{1}{\ell-1}$,
since otherwise there is nothing to prove.
Let $k$ be the smallest index with $1 \le k \le \ell-1$ such that $\vec{p}[k] \neq \vec{q}[k]$.
Since we have assumed that $\phi(\vec{p}) < \phi(\vec{q})$,
it follows that $\vec{p}[k] = 0$ and $\vec{q}[k] = 1$.
By the minimality of $k$, we also have $\vec{p}\slice{1}{k-1} = \vec{q}\slice{1}{k-1}$.
Thus, \ref{enu:diff1} and \ref{enu:diff2} are established.

We now prove \ref{enu:diff3} by contradiction.
Assume that there exists $k < k' \le \ell-1$ such that 
$\vec{p}[k'] = 0$ or $\vec{q}[k'] = 1$.
Then, using \ref{enu:diff1} and \ref{enu:diff2}, we have
\begin{align*}
\phi(\vec{q}) - \phi(\vec{p})
= 2^{-k} + \sum_{i=k+1}^m (\vec{q}[i] - \vec{p}[i])2^{-i}
\ge 2^{-k} + 2^{-k'} - \sum_{i=k+1}^m 2^{-i}
= 2^{-k'} + 2^{-m},
\end{align*}
which contradicts the assumption that $\phi(\vec{q}) - \phi(\vec{p}) < 2^{-\ell+1}$.

The proofs of \ref{enu:diff4} and \ref{enu:diff5} are similar to that of \ref{enu:diff3} and are omitted.
\end{proof}

In particular, this lemma implies the following corollary.

\begin{corollary}\label{cor:01vec}
Let $\ell,m,p,q$ be integers with $2 \le \ell \le m$ and
$0 \le p \neq q < 2^m$, and assume that
$|\phi(\vec{q}) - \phi(\vec{p})| < 2^{-\ell+1}$.
Then the following statements hold:

\begin{enumerate}[label=\textup{(\roman*)}]
\item \label{item:01vec-1}
The vector $(\vec{p}\oplus\vec{q})\slice{1}{\ell-1}$ is either
$\bszero$, $\bsone$, or of the form $(0,\dots,0,1,\dots,1)^\top$.

\item \label{item:01vec-2}
If there exists an integer $2 \le k \le \ell-1$ such that
$(\vec{p} \oplus \vec{q})[k-1] = 0$ and $(\vec{p} \oplus \vec{q})[k] = 1$, 
then
\[
(\vec{p} \oplus \vec{q})\slice{1}{k-1} = \bszero, \quad
(\vec{p} \oplus \vec{q})\slice{k}{\ell-1} = \bsone, \quad
(\vec{p} \oplus \vec{q})\slice{\ell}{m} \neq \bszero.
\]
Moreover, one of the following holds:
\begin{itemize}
\item $\phi(\vec{p}) > \phi(\vec{q})$, 
      $\vec{p}[k] = 1$, $\vec{q}[k] = 0$, 
      $\vec{p}\slice{k+1}{\ell-1} = \bszero$, $\vec{q}\slice{k+1}{\ell-1} = \bsone$, 
      $\vec{p}[\ell] \le \vec{q}[\ell]$;
\item $\phi(\vec{p}) < \phi(\vec{q})$, 
      $\vec{p}[k] = 0$, $\vec{q}[k] = 1$, 
      $\vec{p}\slice{k+1}{\ell-1} = \bsone$, $\vec{q}\slice{k+1}{\ell-1} = \bszero$, 
      $\vec{p}[\ell] \ge \vec{q}[\ell]$.
\end{itemize}

\item \label{item:01vec-3}
If there exists an integer $2 \le k \le \ell-1$ such that
$(\vec{p} \oplus \vec{q})[k-1] = (\vec{p} \oplus \vec{q})[k] = 1$, 
then
\[
(\vec{p}\oplus \vec{q})\slice{k}{\ell-1} = \bsone, \quad
(\vec{p} \oplus \vec{q})\slice{\ell}{m} \neq \bszero.
\]
Moreover, one of the following holds:
\begin{itemize}
\item $\phi(\vec{p}) > \phi(\vec{q})$, 
      $\vec{p}\slice{k}{\ell-1} = \bszero$, $\vec{q}\slice{k}{\ell-1} = \bsone$, 
      $\vec{p}[\ell] \le \vec{q}[\ell]$;
\item $\phi(\vec{p}) < \phi(\vec{q})$, 
      $\vec{p}\slice{k}{\ell-1} = \bsone$, $\vec{q}\slice{k}{\ell-1} = \bszero$, 
      $\vec{p}[\ell] \ge \vec{q}[\ell]$.
\end{itemize}
\end{enumerate}
\end{corollary}

\subsection{Proof of the main theorem}\label{sec:proof-of-main}

\begin{itemize}
\item
The case $m = 2^v - 1$ is essentially treated in \cite{G24a}.
As shown in the proof of \cite[Theorem~2.2]{G24a}, we have
\[
\|\bsx_{1} - \bsx_{2^m-2} \|_\infty = 2^{-m}.
\]
On the other hand, by the construction of the Sobol' sequence,
$\|\bsx_{p} - \bsx_{q} \|_\infty \ge 2^{-m}$ holds for any $p \neq q$.
Combining these observations, we conclude that
\[
q_\infty(Q_{2^m}) = 2^{-m-1}.
\]

\item
The case $m = 2^v$ is essentially treated in \cite{DGSxx}.
From the proof of \cite[Theorem~4.2]{DGSxx} with $L$ being the identity matrix, we deduce
$\bsx_{2^{m-1}+1} = (1/2 + 1/2^m, 1/2 - 1/2^m)$,
which implies
$\|\bsx_{1} - \bsx_{2^{m-1}+1} \|_\infty = 2^{-m}$.
Hence, in the same manner as the previous case, we obtain
\[
q_\infty(Q_{2^m}) = 2^{-m-1}.
\]

\item
For the remaining cases, we exclude $m = 2^v$ or $m = 2^v - 1$ for any $v \in \NN$.
We may then write
\[
m = 2^v + 2^w + c, \quad v \ge 2, \quad v > w, \quad 2^w > c \ge 0,
\]
and set $V = 2^v$, $W = 2^w$.
Since $m+1$ is not a power of two, we have $m \le 2V-2$ and can thus apply
Lemma~\ref{lem:Pascal} \ref{item:Pascal12} and \ref{item:Pascal13}.
In the following, we will establish separately that
\[
q_\infty(Q_{2^m}) \le 2^{-V-W} \quad \text{and} \quad 2^{-V-W} \le q_\infty(Q_{2^m}).
\]
\end{itemize}

\subsubsection{Proof of the upper bound}\label{sec:proof-upper}

Let $p = 2^{V+W-1} + 2^{V-1}$ and $q = 2^{V+W} - 2^{V}$.
To prove $q_\infty(Q_{2^m}) \le 2^{-V-W}$, it suffices to show that
$\|\bsx_p - \bsx_q\|_\infty = 2^{-V-W+1}$.

First, we compute $\bsx_p = (x_{p,1}, x_{p,2})$.
Since $\vec{p}[i] = 1$ if and only if $i=V$ or $i=V+W$, it follows from Lemma~\ref{lem:Pascal} \ref{item:Pascal1} and \ref{item:Pascal11} that
\begin{align*}
x_{p,1} &= \phi(\vec{p}) = 2^{-V} + 2^{-V-W},\\
x_{p,2} &= \phi(P_m \vec{p}) = \sum_{i=1}^{V+W} (P[i][V] \oplus P[i][V+W])2^{-i}
= \sum_{i=W+1}^{V+W} 2^{-i}
= 2^{-W} - 2^{-V-W}.
\end{align*}

Next, we compute $\bsx_q = (x_{q,1}, x_{q,2})$.
Since $\vec{q}[i] = 1$ for $V+1 \le i \le V+W$, we have
\begin{align*}
P_m \vec{q}
&= P_m (\mathrm{concat}(\bsone_{V}, \bszero_{m-V}) \oplus \mathrm{concat}(\bsone_{V+W}, \bszero_{m-V-W}))\\
&= \mathrm{concat}(P_{V} \bsone_{V}, \bszero_{m-V}) \oplus \mathrm{concat}(P_{V+W} \bsone_{V+W}, \bszero_{m-V-W}),
\end{align*}
where ``concat'' denotes vertical concatenation of column vectors.
Thus, by Lemma~\ref{lem:Pascal} \ref{item:Pascal4} and \ref{item:Pascal5}, we obtain
\begin{align*}
x_{q,1} &= \phi(\vec{q}) = \sum_{i=V+1}^{V+W} 2^{-i} = 2^{-V} - 2^{-V-W},\\
x_{q,2} &= \phi(P_m \vec{q}) = 2^{-W} + 2^{-V-W}.
\end{align*}

Hence,
\[
|x_{p,1} - x_{q,1}| = |x_{p,2} - x_{q,2}| = 2^{-V-W+1},
\]
which gives the desired result. \qed

\subsubsection{Proof of the lower bound}\label{sec:proof-lower}

We prove the bound by contradiction. 
Assume that there exist integers $p, q$ with $0 \le p \neq q < 2^m$ such that
$\phi(\vec{p}) < \phi(\vec{q})$ and 
$\|\bsx_p - \bsx_q\|_\infty < 2^{-V-W+1}$.
Let $\vec{\Delta} := \vec{p} \oplus \vec{q}$.
We divide the analysis into three cases according to the values of $\vec{\Delta}[V]$ and $\vec{\Delta}[V-1]$.

\medskip
\noindent\textbf{Case 1: $\vec{\Delta}[V] = 0$.}

\noindent
In this case, Lemma~\ref{lem:Pascal} \ref{item:Pascal12} gives
\[
(P\vec{\Delta})[V] = \vec{\Delta}[V] = 0.
\]
Then, Corollary~\ref{cor:01vec} \ref{item:01vec-1} with $\ell=V+1$ implies
\[
\vec{\Delta}\slice{1}{V} = (P\vec{\Delta})\slice{1}{V} = \bszero.
\]
Hence, $\bsx_{p}$ and $\bsx_{q}$ lie in the same interval of the form
$[a/2^V,(a+1)/2^V) \times [b/2^V,(b+1)/2^V)$
for some $a,b$ with $0 \le a,b < 2^V$.
By Proposition~\ref{prop:0ms}, this forces $\bsx_p = \bsx_q$, contradicting the assumption $p \neq q$.

\medskip
\noindent\textbf{Case 2: $\vec{\Delta}[V] = 1$ and $\vec{\Delta}[V-1] = 0$.}

\noindent
In this case, Corollary~\ref{cor:01vec} \ref{item:01vec-2} with $k=V$ and $\ell=V+W$, together with the assumption $\phi(\vec{p}) < \phi(\vec{q})$, implies
\begin{align}
\vec{\Delta}\slice{V}{V+W-1} = \bsone, \label{eq:case2-1}\\
\vec{\Delta}\slice{V+W}{m} \neq \bszero, \label{eq:case2-2}\\
\vec{p}[V] = 0, \quad \vec{q}[V] = 1, \label{eq:case2-3}\\
\vec{p}[V+W] \ge \vec{q}[V+W]. \label{eq:case2-4}
\end{align}
By Lemma~\ref{lem:Pascal} \ref{item:Pascal12} and \eqref{eq:case2-3}, we have
\[
(P\vec{p})[V] = \vec{p}[V] = 0, \quad
(P\vec{q})[V] = \vec{q}[V] = 1, \quad
(P\vec{\Delta})[V] = \vec{\Delta}[V] = 1,
\]
and from Lemma~\ref{lem:Pascal} \ref{item:Pascal13}, 
\[
(P\vec{\Delta})[V-1] = \vec{\Delta}[V-1] \oplus \vec{\Delta}[V] = 1.
\]
Hence, since $(P\vec{p})[V] = 0$ holds, the first alternative of Corollary~\ref{cor:01vec} \ref{item:01vec-3} gives
\begin{align}
(P\vec{\Delta})\slice{V}{V+W-1} &= \bsone, \label{eq:case2-5}\\
(P\vec{p})[V+W] &\le (P\vec{q})[V+W]. \label{eq:case2-6}
\end{align}

We further divide the analysis into the following two subcases.

\medskip
\noindent\textbf{Case 2-1: $\vec{\Delta}[V+W] = 0$.}

\noindent
In this case, using \eqref{eq:case2-5}, \eqref{eq:case2-1}, Lemma~\ref{lem:Pascal} \ref{item:Pascal2}, and Lemma~\ref{lem:Pascal} \ref{item:Pascal6}, we have
\begin{align*}
\bsone
&= (P\vec{\Delta})\slice{V+1}{V+W-1} \\
&= P\slice{V+1}{V+W-1}\slice{V+1}{V+W-1} \cdot \vec{\Delta}\slice{V+1}{V+W-1} \\
&\quad \oplus P\slice{V+1}{V+W-1}\slice{V+W}{V+W} \cdot \vec{\Delta}[V+W] \\
&\quad \oplus P\slice{V+1}{V+W-1}\slice{V+W+1}{m} \cdot \vec{\Delta}\slice{V+W+1}{m} \\
&= P_{W-1} \bsone \oplus \bszero \oplus P' \cdot \vec{\Delta}\slice{V+W+1}{m} \\
&= \bsone \oplus P' \cdot \vec{\Delta}\slice{V+W+1}{m}.
\end{align*}
where $P' := P\slice{V+1}{V+W-1}\slice{V+W+1}{m}$.
This gives
\begin{equation} \label{eq:P'Delta=0}
P' \cdot \vec{\Delta}\slice{V+W+1}{m} = \bszero.    
\end{equation}
From Lemma~\ref{lem:Pascal} \ref{item:Pascal2}, we have
\begin{align*}
P'
= P\slice{V+1}{V+W-1}\slice{V+W+1}{m}
&= P\slice{1}{W-1}\slice{W+1}{m-V}\\
&= P\slice{1}{W-1}\slice{1}{m-V-W}.    
\end{align*}
Since $P_{m-V-W}$ is non-singular and $W-1 \ge m-V-W$, the columns of $P'$ are linearly independent.
Hence, \eqref{eq:P'Delta=0} implies
\[
\vec{\Delta}\slice{V+W+1}{m} = \bszero,
\]
which together with the assumption $\vec{\Delta}[V+W] = 0$ contradicts \eqref{eq:case2-2}.

\medskip
\noindent\textbf{Case 2-2: $\vec{\Delta}[V+W] = 1$.}

\noindent
In this case, \eqref{eq:case2-4} implies that $\vec{p}[V+W] = 1$ and $\vec{q}[V+W] = 0$.
Hence, Lemma~\ref{lem:Pascal} \ref{item:Pascal12} implies
\[
(P\vec{p})[V+W] = \vec{p}[V+W] = 1, \quad \text{and}
\quad (P\vec{q})[V+W] = \vec{q}[V+W] = 0,
\]
which contradicts \eqref{eq:case2-6}.

\medskip
\noindent\textbf{Case 3: $\vec{\Delta}[V] = 1$ and $\vec{\Delta}[V-1] = 1$.}

\noindent
In this case, Corollary~\ref{cor:01vec} \ref{item:01vec-3}, combined with the assumption $\phi(\vec{p}) < \phi(\vec{q})$, implies
\begin{align}
\vec{\Delta}\slice{V}{V+W-1} = \bsone, \label{eq:case3-1}\\
\vec{\Delta}\slice{V+W}{m} \neq \bszero, \label{eq:case3-2}\\
\vec{p}[V] = 1, \vec{q}[V] = 0, \label{eq:case3-3}\\
\vec{p}[V+W] \ge \vec{q}[V+W]. \label{eq:case3-4}
\end{align}
By Lemma~\ref{lem:Pascal} \ref{item:Pascal12} and \eqref{eq:case3-3}, we have
\[
(P\vec{p})[V] = \vec{p}[V] = 1, \quad
(P\vec{q})[V] = \vec{q}[V] = 0, \quad
(P\vec{\Delta})[V] = \vec{\Delta}[V] = 1,
\]
and from Lemma~\ref{lem:Pascal} \ref{item:Pascal13},
\[
(P\vec{\Delta})[V-1] = \vec{\Delta}[V-1] \oplus \vec{\Delta}[V] = 0.
\]
Hence, since $(P\vec{p})[V] = 1$ holds, the first alternative of Corollary~\ref{cor:01vec} \ref{item:01vec-2} gives
\begin{align}
(P\vec{\Delta})\slice{V}{V+W-1} &= \bsone, \label{eq:case3-5}\\
(P\vec{p})[V+W] &\le (P\vec{q})[V+W]. \label{eq:case3-6}
\end{align}

We now split the analysis into the following two subcases.

\medskip
\noindent\textbf{Case 3-1: $\vec{\Delta}[V+W] = 0$.}

\noindent
In this case, in the same way as in the proof of Case 2-1, \eqref{eq:case3-1} and \eqref{eq:case3-5} imply
\[
\vec{\Delta}\slice{V+W+1}{m} = \bszero.
\]
This, together with the assumption $\vec{\Delta}[V+W] = 0$, contradicts \eqref{eq:case3-2}.

\medskip
\noindent\textbf{Case 3-2: $\vec{\Delta}[V+W] = 1$.}

\noindent
Here, \eqref{eq:case3-4} implies $\vec{p}[V+W] = 1$ and $\vec{q}[V+W] = 0$.
Hence, Lemma~\ref{lem:Pascal} \ref{item:Pascal12} implies
\[
(P\vec{p})[V+W] = \vec{p}[V+W] = 1, \quad \text{and}
\quad (P\vec{q})[V+W] = \vec{q}[V+W] = 0,
\]
which contradicts \eqref{eq:case3-6}.

\medskip
The proof is therefore complete in all cases. \qed

\subsection{Proof of Corollary~\ref{cor:net}}
The cases $m = 1$ and $m = 5$ hold individually, as in Theorem~\ref{thm:main}.

If $m = 2^v$ or $2^v-1$ for some $v \in \NN$ and $m \neq 1$,
then Theorem~\ref{thm:main} gives $q_\infty(Q_{2^m}) = 2^{-m-1}$, and hence
\[
2^{3m/4}q_\infty(Q_{2^m}) = 2^{-m/4-1} \le 2^{-3/2}.
\]

Otherwise, write $m = 2^v + 2^w + c$ with $v > w$ and $2^w > c \ge 0$.
Then Theorem~\ref{thm:main} gives $q_\infty(Q_{2^m}) = 2^{-2^v - 2^w}$.

First, consider $w \le v-2$. Since $m \neq 5$, we have $v \ge 3$. 
Using $c \le 2^w-1$, we obtain
\[
\log_2(2^{3m/4}q_\infty(Q_{2^m}))
= \frac{3}{4}(2^v + 2^w + c) - 2^v - 2^w 
\le -\frac{2^v}{4} + \frac{2^w}{2} - \frac{3}{4}.
\]
Further, using $w \le v-2$, we have $2^w/2 \le 2^{v-3} = 2^v/8$, so that
\[
\log_2(2^{3m/4}q_\infty(Q_{2^m})) 
\le -\frac{2^v}{4} + \frac{2^v}{8} - \frac{3}{4} 
= -\frac{2^v}{8} - \frac{3}{4} < -\frac{3}{2}.
\]

Next, consider $w = v-1$. In this case, $c \le 2^w-2$; otherwise $m+1$ would be a power of two. Then
\begin{align*}
\log_2(2^{3m/4}q_\infty(Q_{2^m}))
&= \frac{3}{4}(2^v + 2^w + c) - 2^v - 2^w \\
&\le \frac{3}{4}(2^{w+1} + 2^w + 2^w - 2) - 2^{w+1} - 2^w \\
&= -\frac{3}{2},
\end{align*}
with equality if $c = 2^w-2$.

This completes the proof in all cases. \qed

\subsection{Proof of Corollary~\ref{cor:limsup}}
To prove \eqref{eq:limsup}, let $N \ge 2$ and choose $m \in \NN$ such that $2^m \le N \le 2^{m+1}$.
Since $q_\infty(Q_N)$ is non-increasing in $N$, \eqref{eq:cor-net-general} gives
\[
N^{3/4} q_\infty(Q_N)
\le (2^{m+1})^{3/4} q_\infty(Q_{2^m})
\le (2^{m+1})^{3/4} \cdot 2^{-3m/4-5/4}
= 2^{-1/2}.
\]
This proves \eqref{eq:limsup} for general $N$.

If $N \ge 64$, then $m \ge 6$, and we can use \eqref{eq:cor-net-strong} instead of \eqref{eq:cor-net-general}; 
the same analysis then gives the improved constants for $q_\infty(Q_N)$.

Finally, \eqref{eq:meshratio} follows immediately from \eqref{eq:limsup} together with the general bound
$h_\infty(Q_N) \ge 1/(2\sqrt{N})$
as given in \cite[Remark~2.4]{DGSxx}.
\qed

\section{Separation radius for general \texorpdfstring{$N$}{N}}\label{sec:arbitraryN}

Building on our previous result for the exact separation radius $q_\infty(Q_{2^m})$ for a power-of-two number of points, this section establishes an upper bound for $q_\infty(Q_{N})$
for an arbitrary number of points $N$.
It is important to note that the analysis herein applies exclusively to the Sobol' sequence in radical inverse order (Definition~\ref{def:Sobol}).
To motivate our theoretical result, we first conducted a numerical experiment to identify the pairs of indices $(p,q)$ that successively update the minimum distance $\|\bsx_p - \bsx_q\|_\infty$ for $q \le 2^{15}$.
Note that, since the Sobol' sequence starts from index~$0$, the point set corresponding to $(p,q)$ consists of $N=q+1$ points.
We recall that the separation radius is the half of the minimum distance.
The results of this search, summarized in Table~\ref{table:radius-update}, motivate the following theorem, which provides an analytical result consistent with our numerical observations.

\begin{table}[b]
  \centering
  \caption{Indices $(p,q)$ for which the distance $\|\bsx_p - \bsx_q\|_\infty$ successively achieves the minimum distance for the Sobol' sequence, shown for $q \le 2^{15}$. The fourth row indicates the corresponding case in Theorem~\ref{thm:decrease_candidates} that describes the pair. The last two rows show the corresponding number of points $N=q+1$ and the separation radius $q_\infty(Q_N) = \|\bsx_p - \bsx_q\|_\infty/2$, respectively.}
\label{table:radius-update}
\begin{tabular}{r|ccccccccccccc}
$p$ & $0$ & $1$ & $2$ & $1$ & $24$ & $40$ & $8$ & $1$ & $504$ & $640$ & $2176$ & $128$ & \\
\hline
$q$ & $1$ & $2$ & $4$ & $9$ & $36$ & $48$ & $112$ & $129$ & $516$ & $768$ & $3840$ & $32512$ & \\
\hline
$\|\bsx_p - \bsx_q\|_\infty$ & $2^{-1}$ & $2^{-2}$ & $2^{-3}$ & $2^{-4}$ & $3 \cdot 2^{-6}$ & $2^{-5}$ & $2^{-7}$ & $2^{-8}$ & $3 \cdot 2^{-10}$ & $2^{-9}$ & $2^{-11}$ & $2^{-15}$ & \\
\hline
Thm.~\ref{thm:decrease_candidates} & --- & --- &\ref{item:dc2}&\ref{item:dc1}&\ref{item:dc3}&\ref{item:dc4}&\ref{item:dc2}&\ref{item:dc1}&\ref{item:dc3}&\ref{item:dc4}&\ref{item:dc4}&\ref{item:dc2}\\
\hline
$N$ & $2$ & $3$ & $5$ & $10$ & $37$ & $49$ & $113$ & $130$ & $517$ & $769$ & $3841$ & $32513$ & \\
\hline
$q_\infty(Q_N)$ & $2^{-2}$ & $2^{-3}$ & $2^{-4}$ & $2^{-5}$ & $3 \cdot 2^{-7}$ & $2^{-6}$ & $2^{-8}$ & $2^{-9}$ & $3 \cdot 2^{-11}$ & $2^{-10}$ & $2^{-12}$ & $2^{-16}$ & \\
\end{tabular}

\end{table}

\begin{theorem} \label{thm:decrease_candidates}
Let $Q = \{\bsx_0, \bsx_1, \dots \}$ be the two-dimensional Sobol' sequence and $Q_N$ be its first $N$ points. 
Let $v > w \ge 1$ be integers.
Then we have the following.
\begin{enumerate}[label=\textup{(\roman*)}]
\item \label{item:dc1}
\textbf{Case where $m=2^v$:}
Let $p=1$, $q=2^{2^v-1}+1$, and $N=q+1$.
Then 
\[q_\infty(Q_N) \le \dfrac{1}{2}\|\bsx_p - \bsx_q\|_\infty = 2^{-2^v-1}\]

\item \label{item:dc2}
\textbf{Case where $m=2^v-1$:}
Let $p=2^{2^{v-1}-1}$, $q=2^{2^v-1} - 2^{2^{v-1}}$, and $N=q+1$.
Then
\[q_\infty(Q_N) \le \dfrac{1}{2}\|\bsx_p - \bsx_q\|_\infty = 2^{-2^v}.\]

\item \label{item:dc3}
\textbf{Case where $m=2^v+2$:}
Let $p=2^{2^v+1}-8$, $q=2^{2^v+1}+4$, and $N=q+1$.
Then
\[q_\infty(Q_N) \le \dfrac{1}{2}\|\bsx_p - \bsx_q\|_\infty = 3 \cdot 2^{-2^v-3}\]

\item \label{item:dc4}
\textbf{Case where $m=2^v+2^w$:}
Let $p=2^{2^v+2^w-1}+2^{2^v-1}$, $q=2^{2^v+2^w} - 2^{2^v}$, and $N=q+1$.
Then
\[q_\infty(Q_N) \le \dfrac{1}{2}\|\bsx_p - \bsx_q\|_\infty = 2^{-2^v-2^w}.\]
\end{enumerate}
\end{theorem}

\begin{proof}

In this proof, we use the notation $V = 2^v$ and $W =2^w$
as in Section~\ref{sec:proofs}.

\ref{item:dc1}:
This is shown in the second case of Section~\ref{sec:proof-of-main}.

\ref{item:dc2}:
Let $V' = 2^{v-1}$.
First, we compute $\bsx_p = (x_{p,1}, x_{p,2})$.
Since $\vec{p}[i] = 1$ if and only if $i = V'$ ,
it follows from Lemma~\ref{lem:Pascal} \ref{item:Pascal1} that
\begin{align*}
x_{p,1} &= \phi(\vec{p}) = 2^{-V'}, \\
x_{p,2} &= \phi(P_{V-1} \vec{p})
= \sum_{i=1}^{V'} P[i][V'] 2^{-i}
= \sum_{i=1}^{V'} 2^{-i}
= 1 - 2^{-V'}.
\end{align*}

Next, we compute $\bsx_q = (x_{q,1}, x_{q,2})$.
In what follows, ``concat'' denotes vertical concatenation of column vectors.
Since $\vec{q}[i] = 1$ for $V'+1 \le i \le V-1$, we have
\begin{align*}
P_{V-1} \vec{q}
&= P_{V-1} (\bsone_{V-1} \oplus \mathrm{concat}(\bsone_{V'}, \bszero_{V-V'-1}))\\
&= P_{V-1}\bsone_{V-1} \oplus \mathrm{concat}(P_{V'}\bsone_{V'}, \bszero_{V-V'-1}))\\
&= \bsone_{V-1} \oplus \text{(the vector whose $V'$-th component is $1$ and all others are $0$)}\\
&= \mathrm{concat}(\bsone_{V'-1}, \bszero_1, \bsone_{V-V'-1}),
\end{align*}
where the third equality follows from
Lemma~\ref{lem:Pascal} \ref{item:Pascal4}
and \ref{item:Pascal6}.
Thus we obtain
\begin{align*}
x_{q,1} &= \phi(\vec{q})
= \sum_{i=V'+1}^{V-1} 2^{-i} = 2^{-V'} - 2^{-V+1}\\
x_{q,2} &= \phi(P_{V-1} \vec{q})
= \sum_{i=1}^{V-1} 2^{-i} - 2^{-V'} 
= 1 - 2^{-V'} - 2^{-V+1}.
\end{align*}
Hence $\|\bsx_p - \bsx_q\| = 2^{-V+1}$.

\ref{item:dc3}:
First, we compute $\bsx_p = (x_{p,1}, x_{p,2})$.
We have $\vec{p}[i] = 1$ for $4 \le i \le V+1$.
From Lemma~\ref{lem:Pascal} \ref{item:Pascal5} with $(V,W) = (2,1)$ and $(V,1)$,
we have
\begin{align*}
(P_3 \bsone_{3}) [i] &= 1 \iff i = 1,2,3\\
(P_{V+1} \bsone_{V+1}) [i] &= 1 \iff i = 1,V,V+1.
\end{align*}
Since we have assumed that $v \ge 2$ and thus $V \ge 4$, we have
\begin{align*}
(P_{V+2} \vec{p})[i] = 1 \iff i=2,3,V,V+1,
\end{align*}
and thus
\begin{align*}
x_{p,1} &= \phi(\vec{p}) = \sum_{i=4}^{V+1}2^{-i} = 2^{-3} - 2^{-V-1}, \\
x_{p,2} &= \phi(P_{V+2} \vec{p})
= 2^{-2} + 2^{-3} + 2^{-V} + 2^{-V-1}.
\end{align*}

Next, we compute $\bsx_q = (x_{q,1}, x_{q,2})$.
Since $\vec{q}[i] = 1$ if and only if $i=3$ or $i=V+2$,
it follows from Lemma~\ref{lem:Pascal} \ref{item:Pascal11} with $(V,W) = (2,1)$ and $(V,2)$ that
\begin{align*}
x_{q,1} &= \phi(\vec{q})
= 2^{-3} + 2^{-V-2}\\
x_{q,2} &= \phi(P_{V+2} \vec{q})
= \sum_{i=1}^{V+2} (P[i][3] \oplus P[i][V+2]) 2^{-i}
= 2^{-2} + 2^{-3} + 2^{-V-1} + 2^{-V-2}.
\end{align*}
Hence $\|\bsx_p - \bsx_q\| = 3 \cdot 2^{-V-2}$.

\ref{item:dc4}: This is shown in Section~\ref{sec:proof-upper}.
\end{proof}

Establishing a lower bound on the separation radius or minimum distance for general $N$ would necessitate a significantly more complex argument than the one used for the proof in Section~\ref{sec:proof-lower}.
We therefore propose the conjecture that, for $q \ge 4$, a pair $(\bsx_p, \bsx_q)$ updates the minimum distance if and only if it is one of the candidates specified in Theorem~\ref{thm:decrease_candidates}.

\section{Conclusion}
We established a complete characterization of the separation radius for the case $N=2^m$ and investigated its numerical behavior for general $N$.
Our analysis reveals that the two-dimensional Sobol' sequence has a suboptimal mesh ratio that grows at least as 
$N^{1/4}$.
While digital scrambling may alleviate the poor growth of the separation radius, examining this possibility is left for future work.

In sharp contrast to the two-dimensional case (Theorem~\ref{thm:main}), Sobol' points in dimensions $d \ge 3$ did not exhibit a clear systematic pattern in their separation radius.
For $d=3$, in particular, the observed rate lies between 
$N^{-1/3}$ and $N^{-1/2}$ up to $N \le 2^{15}$.
Understanding the behavior in higher dimensions and establishing dimension-dependent bounds remain important directions for future research.

\section*{Acknowledgments}
The author acknowledge the use of OpenAI's ChatGPT 5 and Google's Gemini 2.5 for editing and language polishing of the manuscript.

\bibliographystyle{amsplain}
\bibliography{ref.bib}

\end{document}